\documentclass[12pt,en]{elegantpaper}
\usepackage{esint}
\usepackage{extarrows}
\usepackage{mathrsfs}
\numberwithin{equation}{section}
\allowdisplaybreaks[4]
\everymath{\displaystyle}

\title{Liouville theorem for a class of p-Laplace type equations on manifolds}

\author{Xi-Nan Ma\and Wei Wei\and Tian Wu \and Hua Zhu}

\begin{document}

\date{}
\maketitle

\renewcommand{\thefootnote}{\fnsymbol{footnote}}
\newcommand{\Red}[1]{\textcolor{red}{#1}}
\newcommand{\Blue}[1]{\textcolor{blue}{#1}}
\newcommand{\Green}[1]{\textcolor{green}{#1}}


\begin{abstract}
    \setlength{\abovedisplayskip}{2pt}
    \setlength{\belowdisplayskip}{2pt}
    \setlength{\abovedisplayshortskip}{0pt}
    \setlength{\belowdisplayshortskip}{0pt}
    \hspace{1em}We study a class of $p$-Laplace equations
    $$\Delta_p u-\lambda u^{p-1}+ u^{q-1}=0$$
    on a closed $n$-dimensional Riemannian manifold $(M,g)$ with $\operatorname{Ric}\geqslant(n-1)g$. For
    $1<p<2$, $p<q<p^*$, and $0<\lambda<S_{p,q}^{-1}$, where
    $$S_{p,q}=\frac{q-p}{2}(\frac p n)^{\frac p 2}\Big(\frac{(p^*-1)^2(2-p)}{(p_*-1)(q-1)(p^*-q)}\Big)^{\frac{2-p}{2}},$$
    with $p_*=\frac{(n-1)p}{n-p}$ and $p^*=\frac{np}{n-p}$, we prove that the constant $\lambda^{\frac{1}{q-p}}$ is the unique positive solution of the equation. In contrast, for $p>2$ and $p<q<p^*$, the uniqueness fails for every $\lambda>0$; aside from the constant solution, the equation admits a positive nonconstant solution. This answers V\'eron's problem raised in \cite{Ver92}.
\end{abstract}

\section{Introduction}\label{sec-introduction}

For the $p$-Laplace operator on a manifold $(M^n,g)$, V\'{e}ron \cite{Ver92} raised the following two open problems:
\begin{itemize}
    \item[(1)] For $1<p<n$ and $p<q<p^*:=\frac{np}{n-p}$, under what conditions on $\lambda$, $q$ and $\gamma>0$, are the only positive solutions to
    \begin{equation}\label{eq:q}
        \Delta_p u-\lambda u^{p-1}+u^{q-1}=0
    \end{equation}
    constant? The $p$-Laplace operator is given by $\Delta_p u:=\mathrm{div}(|\nabla u|^{p-2}\nabla u)$.
    \item[(2)] For $1<p<n$ and $p<q<p^*$, under what conditions on $\lambda$, $q$ and $\gamma>0$, are the only positive solutions to
    $$\operatorname{div}[(\gamma^2 u^2+|\nabla u|^2)^{\frac{p-2}{2}}\nabla u]-\lambda(\gamma^2 u^2+|\nabla u|^2)^{\frac{p-2}{2}}u+\gamma^{p-2}u^{q-1}=0$$
    constant?
\end{itemize}
The second problem was resolved by Lin and Ma \cite{LM2024} on $\mathbb S^n$ under the condition that
$$\gamma=\frac{p}{q-p},\quad\lambda=\gamma[n+1-(q-1)\gamma],$$
which was also conjectured by V\'{e}ron \cite{Ver17}.

In this paper, we answer the first problem. Specifically, we derive a Liouville theorem for $1<p<2$ and $0<\lambda<S_{p,q}^{-1}$, and prove the existence of nonconstant solutions for $p>2$.

\begin{theorem}\label{thm:q}
    Suppose $(M,g)$ is a closed Riemannian manifold of dimension $n\geqslant2$ satisfying $\operatorname{Ric}\geqslant (n-1)g$, $1<p<2$, $p<q<p^*$, and $0<\lambda<S_{p,q}^{-1}$, where
    \begin{equation}\label{cond}
        \quad S_{p,q}=\frac{q-p}{2}(\frac p n)^{\frac p 2}\big(\frac{(p^*-1)^2(2-p)}{(p_*-1)(q-1)(p^*-q)}\big)^{\frac{2-p}{2}}.
    \end{equation}
   Then every positive solution to the equation \eqref{eq:q} is the constant $\lambda^{\frac{1}{q-p}}$.
\end{theorem}

This Liouville theorem implies the $p$-Sobolev inequality. We denote $\fint_M h:=\frac{1}{\operatorname{vol}(M)}\int_M h$.

\begin{corollary}
    Under the hypotheses of Theorem \ref{thm:q},
    \begin{equation}\label{sobolevforp<2}
    \big(\fint_M|v|^q\big)^{\frac p q}\leqslant S_{p,q}\fint_M|\nabla v|^p+\fint_M|v|^p,\quad v\in W^{1,p}(M).
    \end{equation}
    Equality holds in \eqref{sobolevforp<2} if and only if $v$ is constant.
\end{corollary}

In the limit as $p\to2^-$, inequality \eqref{sobolevforp<2} formally reduces to the classical Sobolev inequality:   
$$\Big(\fint_M |v|^q\Big)^{\frac 2 q}\leqslant\frac{q-2}{n}\fint_M |\nabla v|^2+\fint_M v^2,\quad v\in H^1(M),\quad 2<q<2^*=\frac{2n}{n-2}.$$
Furthermore, when specialized to the case $p=2$,  the proof of Theorem \ref{thm:q} recovers the classical semilinear result of Bidaut-V\'eron and V\'eron \cite{BV91} as follows.

\begin{theorem}[\cite{BV91,Oba71}]
    Suppose $(M,g)$ is a closed Riemannian manifold of dimension $n\geqslant3$ satisfying $\operatorname{Ric}\geqslant (n-1)g$, $2<q<2^*$, and $0<\lambda\leqslant\frac{n}{q-2}$. Then every positive solution to the equation $\Delta u-\lambda u+u^{q-1}=0$ is the constant $\lambda^{\frac{1}{q-2}}$.
\end{theorem}

As for the case $p>2$, we obtain the existence of nonconstant solutions via the stability of the Sobolev functional near a constant function.

\begin{theorem}\label{thm:q>2}
    Suppose $(M,g)$ is a closed Riemannian manifold of dimension $n\geqslant3$ satisfying $\operatorname{Ric}\geqslant (n-1)g$, $2<p<n$, and $p<q<p^*$. Then the equation \eqref{eq:q} admits a nonconstant positive solution for every $\lambda>0$.
\end{theorem}

The vector‑field method (differential‑identity method) was first introduced by Obata \cite{Oba71} to classify positive solutions of the Yamabe equation on closed manifolds.  Ma and Wu \cite{MW24} introduced the invariant‑tensor technique to systematically construct vector fields, which  has been applied to various semilinear equations on the Heisenberg group, in Cauchy--Riemann geometry, and for fourth-order equations. Classification of $p$-Laplacian equations is more difficult due to their quasilinear nature, even in $\mathbb R^n$. Serrin and Zou \cite{SZ02} derived a differential identity to prove the nonexistence of positive solutions to $-\Delta_p u=f(u)$ under a subcritical condition on $f$. Recently, Wu and Zhu reformulated this proof by generalizing the invariant tensor technique to the $p$-Laplacian case. After that, by considering a general semilinear term $\alpha(u)$ in the invariant tensor, Liang, Wu, and Yan \cite{LWY25} extended Serrin and Zou's result to an anisotropic case. See \cite{HSW2024, Ou2025} and others for recent developments for the $p$-Laplacian equations.

We now summarize the relevant regularity results for solutions to \eqref{eq:q}. A positive function $u$ is said to be a weak solution to \eqref{eq:q} if $u\in W^{1,p}(M)\cap L^\infty(M)$ satisfies the integral identity
$$\int_M |\nabla u|^{p-2}\langle \nabla u,\nabla\varphi\rangle+\lambda\int_M u^{p-1}\varphi-\int_M u^{q-1}\varphi=0,\quad\forall\varphi\in C^\infty(M).$$
By \cite{ACF23,DiB83,Heb99,Tol84}, we have the following properties:
\begin{enumerate}
    \item \textbf{Global $C^{1,\theta}$ regularity.} By the regularity theory for quasilinear elliptic equations on compact manifolds, there exists $\theta\in(0,1)$ such that $u\in C^{1,\theta}(M)$.

    \item \textbf{Critical set.} Let $\Omega_{\mathrm{cr}} = \{x\in M : \nabla u(x)=0\}$. Then $\Omega_{\mathrm{cr}}$ is a Lebesgue null set of $M$.

    \item \textbf{$H^2$ regularity away from critical points.} On $M\setminus\Omega_{\mathrm{cr}}$, $|\nabla u|$ is uniformly bounded away from zero, and  $u\in H^2(M\setminus\Omega_{\mathrm{cr}})$.
    
    \item \textbf{Global $W^{1,2}$ regularity of the main vector field.} $|\nabla u|^{p-2}\nabla u\in H^1(TM)$.

    \item \textbf{Weighted integrability of Hessian.}  $|\nabla u|^{p-2}\nabla^2 u \in L^2(M\setminus\Omega_{\mathrm{cr}})$.

    \item \textbf{Application in proofs.} All differential identities are first established on $M\setminus\Omega_{\mathrm{cr}}$, then extended to the whole manifold by finite open covering, partition of unity, and mollification. Terms involving connection and curvature are lower-order perturbations.
\end{enumerate}

This paper is structured as follows. In Section \ref{sec:id}, we prove Theorem \ref{thm:q} by establishing the differential identity \eqref{id} based on the invariant tensor $E_{ij}$. In Section \ref{sec:unstable}, we prove Theorem \ref{thm:q>2} by verifying the instability of the Sobolev functional at constant functions.
\section{The Liouville theorem: proof of Theorem \ref{thm:q}}\label{sec:id}

In local coordinates, we denote the metric tensor as $g_{ij}$, and the Ricci curvature tensor $\mathrm{Ric}$ as $R_{ij}$. We use the Einstein summation convention, raising and lowering indices with the metric $g_{ij}$ and its inverse $g^{ij}$. Throughout, $u$ denotes a positive solution to \eqref{eq:q}, and $u_i:=\nabla_i u$, $u_{ij}:=\nabla_j\nabla_i u$. All computations below are performed pointwise away from the critical set $\Omega_{\mathrm{cr}}$, using the regularity theory summarized in Section \ref{sec-introduction}.

Let $c\in\mathbb R$ be a constant to be determined. We define the trace-free tensor
\begin{equation}\label{invariant}
    E_{ij}=\nabla_j(|\nabla u|^{p-2}u_{i})+c\frac{|\nabla u|^{p-2}u_iu_j}{u}-\frac 1 n(\Delta_p u+c\frac{|\nabla u|^p}{u})g_{ij}.
\end{equation}
When $p=2$ and $c=0$, this $E_{ij}$ coincides with the classical traceless tensor in \cite{Oba71}.
Applying Ricci's identity yields
$$\nabla^i E_{ij}=\frac{n-1}{n}\nabla_j\Delta_p u+|\nabla u|^{p-2}R_{ij}u^i-\frac{n-1}{n}c\frac{|\nabla u|^p}{u^2}u_j+c\frac{\Delta_p u}{u}u_j+\frac{n-p}{n}c\frac{|\nabla u|^{p-2}}{u}u_{ij}u^i.$$
Expanding the term $\nabla_j(|\nabla u|^{p-2}u_i)$ in $E_{ij}$ gives
$$E_{ij}u^i=(p-1)|\nabla u|^{p-2}u_{ij}u^i+\frac{n-1}{n}c\frac{|\nabla u|^p}{u}u_j-\frac 1 n\Delta_p u u_j.$$
Differentiating equation \eqref{eq:q} yields $-\nabla_j\Delta_p u=(q-1)u^{q-2}u_j-\lambda (p-1)u^{p-2}u_{j}$. Substituting $E_{ij}u^i$ for $|\nabla u|^{p-2}u_{ij}u^i$ and using the expression for $\nabla_j\Delta_p u$, we obtain
\begin{align}\label{Eiji}
    \begin{split}
        \nabla^i E_{ij}=~&\frac{c}{p_*-1}\frac{E_{ij}u^i}{u}-\frac{n-1}{n}c\big(1+\frac{c}{p_*-1}\big)\frac{|\nabla u|^p}{u^2}u_j+|\nabla u|^{p-2}R_{ij}u^i\\
        &+\frac{n-1}{n}\lambda\big(p-1+\frac{p^*-1}{p_*-1}c\big)u^{p-2}u_j-\frac{n-1}{n}\big(q-1+\frac{p^*-1}{p_*-1}c\big)u^{q-2}u_j,
    \end{split}
\end{align}
To eliminate the term involving $u^{q-2}u_j$, we choose $c=-\frac{p_*-1}{p^*-1}(q-1)$. Combining these computations leads to the following identity.

\begin{lemma}
    Let $E_{ij}$ be defined as \eqref{invariant} and $c=-\frac{p_*-1}{p^*-1}(q-1)$. Then
    \begin{align}\label{id}
        \begin{split}
            &u^{\frac{p_*-2}{p^*-1}(q-1)}\nabla^i(u^{-\frac{p_*-2}{p^*-1}(q-1)}E_{ij}|\nabla u|^{p-2}u^j)\\
            =~&E_{ij}E^{ji}+|\nabla u|^{2p-4}R_{ij}u^iu^j+\frac{n-1}{n}\lambda(p-q)u^{p-2}|\nabla u|^p\\
            &+\frac{n-1}{n}\frac{(p_*-1)(q-1)(p^*-q)}{(p^*-1)^2}\frac{|\nabla u|^{2p}}{u^2}.
        \end{split}
    \end{align}
\end{lemma}
\begin{proof}
    By \eqref{Eiji}, the choice of $c$, and direct computations, we obtain
    \begin{align*}
        \begin{split}
            \nabla^i(E_{ij}|\nabla u|^{p-2}u^j)=~&\nabla^i E_{ij}|\nabla u|^{p-2}u^j+E_{ij}\nabla^i(|\nabla u|^{p-2}u^j)\\
            =~&E_{ij}E^{ji}+\frac{p_*-2}{p^*-1}(q-1)\frac{|\nabla u|^{p-2}}{u}E_{ij}u^iu^j+|\nabla u|^{2p-4}R_{ij}u^iu^j\\
            &+\frac{n-1}{n}\lambda(p-q)u^{p-2}|\nabla u|^p+\frac{n-1}{n}\frac{(p_*-1)(q-1)(p^*-q)}{(p^*-1)^2}\frac{|\nabla u|^{2p}}{u^2},
        \end{split}
    \end{align*}
    where we have used the identity $\nabla^j(|\nabla u|^{p-2}u^i)=E^{ij}-c\frac{|\nabla u|^{p-2}u^iu^j}{u}+\frac 1 n(\Delta_p u+c\frac{|\nabla u|^p}{u})g^{ij}$ in the last equality. Multiplying the vector field by $u^{-\frac{p_*-2}{p^*-1}(q-1)}$ completes the proof.
\end{proof}

Although $E_{ij}$ need not be symmetric, the structure of the tensor implies that $E_{ij}E^{ij}$ is non-negative, see \cite[Lemma 2.4]{LWY25}. For completeness, we include the following proof.
\begin{lemma}\label{|E|positive}
    Under the above notations, we have
    $$\frac{n-1}{n}|\nabla u|^2E_{ij}E^{ji}\geqslant E_{ij}u^jE^{ki}u_k\geqslant 0.$$
\end{lemma}
    
\begin{proof}
    Using geodesic normal coordinates centered at $x_0\in M$, we assume $u_1=|\nabla u|>0$, $g_{ij}=\delta_{ij}$ for $1\leqslant i,j\leqslant n$, $u_{ij}=0$ for $2\leqslant i<j\leqslant n$ at the point $x_0$. All computations below are performed at $x_0$. Then for $2\leqslant i,j\leqslant n$ with $i\neq j$, we have
    $$E_{ij}=0,\quad E_{i1}=(u_1)^{p-2}u_{i1},\quad E_{1i}=(p-1)(u_1)^{p-2}u_{1i}=(p-1)E_{i1}.$$
    
    We first show the non-negativity of the right-hand term:
    $$E_{ij}u^jE^{ki}u_k=u_1^2E_{11}^2+(p-1)u_1^2\sum_{i=2}^nE_{i1}^2\geqslant0.
    $$
    Using $\sum_{i=1}^nE_{ii}=0$, we compute
    \begin{align*}
        \frac{n-1}{n}|\nabla u|^2E_{ij}E^{ji}-E_{ij}u^jE^{ki}u_k=~&\frac{n-1}{n}u_1^2\sum_{i=1}^nE_{ii}^2-u_1^2E_{11}^2+\frac{n-2}{n}(p-1)u_1^2\sum_{i=2}^nE_{i1}^2\\
        \geqslant~&\frac{n-1}{n}u_1^2\sum_{i=2}^nE_{ii}^2-\frac{1}{n}u_1^2\left|\sum_{i=2}^nE_{ii}\right|^2\\
        =~&\frac{1}{n}u_1^2\sum_{2\leqslant i<j\leqslant n}|E_{ii}-E_{jj}|^2\geqslant0,
    \end{align*}
    which completes the proof.
\end{proof}

We now prove Theorem \ref{thm:q} using identity \eqref{id}.

\begin{proof}[Proof of Theorem \ref{thm:q}]
   By Lemma \ref{|E|positive} and the curvature condition $\operatorname{Ric}\geqslant(n-1)g$, we obtain
    \begin{align}\label{ineq:id}
        \begin{split}
            &\frac{n}{n-1}u^{\frac{p_*-2}{p^*-1}(q-1)}\nabla^i(u^{-\frac{p_*-2}{p^*-1}(q-1)}E_{ij}|\nabla u|^{p-2}u^j)\\
            \geqslant~&n|\nabla u|^{2p-2}+\frac{(p_*-1)(q-1)(p^*-q)}{(p^*-1)^2}\frac{|\nabla u|^{2p}}{u^2}-\lambda (q-p)|\nabla u|^pu^{p-2}.
        \end{split}
    \end{align}
    
    Using the Young inequality with $\frac p 2+\frac{2-p}{2}=1$ yields
    $$\lambda (q-p)|\nabla u|^pu^{p-2}\leqslant n|\nabla u|^{2p-2}+\frac{2-p}{2^{\frac{2}{2-p}}}(\frac p n)^{\frac{p}{2-p}}\big(\lambda(q-p)\big)^{\frac{2}{2-p}}\frac{|\nabla u|^{2p}}{u^2}.$$
    The assumption $1<p<2$ is used only here, specifically to ensure $\frac{2-p}{2}>0$. Substituting this into \eqref{ineq:id}, we conclude that
    \begin{align*}
        &\frac{n}{n-1}u^{\frac{p_*-2}{p^*-1}(q-1)}\nabla^i(u^{-\frac{p_*-2}{p^*-1}(q-1)}E_{ij}|\nabla u|^{p-2}u^j)\\
        \geqslant~&\Big(\frac{(p_*-1)(q-1)(p^*-q)}{(p^*-1)^2}-\frac{2-p}{2^{\frac{2}{2-p}}}(\frac p n)^{\frac{p}{2-p}}\big(\lambda(q-p)\big)^{\frac{2}{2-p}}\Big)\frac{|\nabla u|^{2p}}{u^2}.
    \end{align*}
    
    Multiplying the above inequality by $u^{-\frac{p_*-2}{p^*-1}(q-1)}$, the integral of the left-hand side over $M$ vanishes by Stokes' theorem. Since the right-hand side is nonnegative under \eqref{cond}, it must be identically zero. In particular, $|\nabla u|=0$, hence $u$ is constant.
\end{proof}

\section{The instability of the Sobolev functional: proof of Theorem \ref{thm:q>2}}\label{sec:unstable}

In this section, we focus on the case $2<p<n$. The key observation is that the functional
$$I_q[u]=\frac{\fint_M|\nabla u|^p+\lambda\fint_M|u|^p}{\big(\fint_M|u|^q\big)^{\frac p q}}$$
is unstable at the constant function $1$, that is, there exists $\varphi\in W^{1,p}(M)\setminus\{0\}$ such that
\begin{equation}\label{unstable}
    I_q[1+\varepsilon\varphi]<I_q[1]=\lambda\text{ for sufficiently small }|\varepsilon|.
\end{equation}

\begin{proof}[Proof of Theorem \ref{thm:q>2}]
    With the observation above, the remaining proof is routine. For the reader's convenience, we sketch the proof and divide it into two steps.
    
    \vspace{0.5em}
    \noindent\textbf{Step 1: the infimum of $I_q[u]$ is attained by a positive function satisfying \eqref{eq:q}.}
    
    Let $\{u_j\}$ be a minimizing sequence with $\fint_M|u_j|^q =1$ and
    $$\fint_M(|\nabla u_j|^p+\lambda|u_j|^p)=I_q[u_j]\rightarrow \inf_{u\in W^{1,p}(M)\setminus\{0\}}I_q[u].$$
    Since $\lambda>0$, we have the $W^{1,p}(M)$-boundedness of $\{u_j\}$. Thus, by the Rellich–Kondrachov compactness theorem, up to a subsequence, there exists $u\in W^{1,p}(M)$ such that
    $$u_j\rightharpoonup u\quad\text{in }W^{1,p}(M),\quad u_j\to u\quad\text{in }L^q(M),\quad\text{and}\quad\fint_M|u|^q=1.$$
    Moreover, weak lower semicontinuity gives
    $$\fint_M(|\nabla u|^p+\lambda|u|^p)\leqslant\liminf_{j\to\infty}\fint_M(|\nabla u_j|^p+\lambda|u_j|^p).$$
    Consequently, $u$ attains the infimum.

    Replacing $u$ by $|u|$ preserves both the constraint and the energy, since
    $$|\nabla |u||=|\nabla u|\quad\text{almost everywhere},\quad I_q[|u|]=I_q[u].$$
    We may therefore assume that the normalized minimizer satisfies $u\geqslant0$. By standard computations, we obtain the
    Euler--Lagrange equation
    \begin{equation}\label{eq-v}
      -\Delta_pu+\lambda u^{p-1}=\big(\fint_M\big(|\nabla u|^p+\lambda u^p\big)\big)u^{q-1}\quad\text{with }\fint_Mu^q=1.
    \end{equation}   
    
    The subcritical growth $u^{q-1}$ and the standard regularity theory for the
    $p$-Laplacian give a $C^{1,\alpha}$ representative of $u$.  Equation \eqref{eq-v} also implies $-\Delta_pu+\lambda u^{p-1}\geqslant0$. By the strong maximum principle,  $u$ is strictly positive since $\fint_Mu^q =1$. Hence $\inf_{\substack{w\in W^{1,p}(M)\\w\ne0}}I_q[w]=\min_{\substack{w\in W^{1,p}(M)\\w>0}}I_q[w]$.
    
    By rescaling, equation \eqref{eq-v} becomes \eqref{eq:q}.
    
    \vspace{0.5em}
    \noindent\textbf{Step 2: positive constants cannot be a minimizer of $I_q$.}
    
    Fix any $\varphi\in C^\infty(M)$ with $\fint_M\varphi=0$. For $|\varepsilon|\ll1$, we have $1+\varepsilon\varphi>0$. Thus, for $p>2$, 
    \begin{align*}
        I_q[1+\varepsilon\varphi]=~&\frac{\varepsilon^p\fint_M|\nabla \varphi|^p+\lambda\fint_M(1+\varepsilon\varphi)^p}{\big(\fint_M(1+\varepsilon\varphi)^q\big)^{\frac p q}}\\
        =~&\Big(\lambda+\lambda\frac{p(p-1)}{2}\varepsilon^2\fint_M\varphi^2+o(\varepsilon^2)\Big)\Big(1+\frac{q(q-1)}{2}\varepsilon^2\fint_M\varphi^2+o(\varepsilon^2)\Big)^{-\frac p q}\\
        =~&\lambda\Big(1-\frac{p(q-p)}{2}\varepsilon^2\fint_M\varphi^2+o(\varepsilon^2)\Big).
    \end{align*}
    Thus \eqref{unstable} holds by $q>p$, so the constant function $1$ cannot be a local or global minimizer of $I_q$. Since $I_q$ is invariant under the rescaling $u\mapsto cu$, the same conclusion holds for any positive constant function $c$. Together with Step 1, this implies the existence of a nonconstant positive minimizer, and Theorem \ref{thm:q>2} is proved.
\end{proof}

\vspace{1em}
\textbf{Acknowledgments.} 
Xi-Nan Ma and Tian Wu are supported by the National Key Research and Development Project (Grant No. 2025YFA1017600). Xi-Nan Ma is also supported by the National Natural Science Foundation of China (Grant No. 12141105). Part of the work was carried out while Wei Wei was a postdoctor at Fudan University, supported by BoXin Program. She is also partially supported by the NSFC (Grant No. 12571218, 12271244). Tian Wu and Hua Zhu are supported by University of Science and Technology of China-Southwest University of Science and Technology Counterpart Cooperation and Development Joint Fund (Grant No. KY0010002501) and the Open Research Fund of Hubei Key Laboratory of Mathematical Sciences (Central China Normal University, Wuhan 430079, P. R. China). Tian Wu is also supported by the Fundamental Research Funds for the Central Universities (Grant No. WK0010250106). Hua Zhu is also supported by the National Natural Science Foundation of China (Grant No. 12501273) and the Research Foundation of Southwest University of Science and Technology (Grant No. 25zx7153).

\noindent \textbf{Research ethics:} Not applicable.

\noindent \textbf{Informed consent:} Not applicable.

\noindent \textbf{Author contributions:} All authors have accepted responsibility for the entire content of this manuscript and approved its submission.

\noindent \textbf{Use of Large Language Models, AI and Machine Learning Tools:} None declared.

\noindent \textbf{Conflict of interest:} The authors state no conflict of interest.

\noindent \textbf{Data availability:} Not applicable.


\footnotesize{
    Welcome contact us:
    \begin{itemize}
        \item Xi-Nan Ma, School of Mathematical Sciences, University of Science and Technology of China, Hefei, Anhui, 230026, People's Republic of China. Email: \emph{xinan@ustc.edu.cn}
        \item Wei Wei, School of Mathematics, Nanjing University, Nanjing, Jiangsu, 210093, People's Republic of China. Email: \emph{wei\_wei@nju.edu.cn}
        \item Tian Wu, School of Mathematical Sciences, University of Science and Technology of China, Hefei, Anhui, 230026, People's Republic of China. Email: \emph{wt1997@ustc.edu.cn}
        \item Hua Zhu, School of Mathematical and Physics, Southwest University of Science and Technology, Mianyang, Sichuan, 621010, People's Republic of China. Email: \emph{zhuhmaths@mail.ustc.edu.cn}
    \end{itemize}
}

\end{document}